\documentclass[a4paper,12pt]{elsarticle}
\usepackage[a4paper,left=2cm,right=2cm,top=3cm,bottom=2.5cm]{geometry}
\usepackage{setspace}
\usepackage{amsmath,amsthm,amsfonts,amssymb}
\usepackage{latexsym}
\usepackage{xcolor} %para poder chamar desenhos
\usepackage{graphicx} %também para os desenhos
\usepackage{bm}
\usepackage{makeidx} %%% fazer índice remissivo%%%
\usepackage{epigraph}
\usepackage[Sonny]{fncychap}
\usepackage[all]{xy}

\usepackage{graphicx}
\usepackage{amssymb}
\usepackage{lineno}

\newcommand{\R}{\mathbb{R}}

\newcommand{\C}{\mathbb{C}}

\newcommand{\dis}{\displaystyle}

\newtheorem{theorem}{Theorem}

\newtheorem{lemma}{Lemma}

\newtheorem{definition}{Definition}

\newtheorem{remark}{Remark}
\journal{??????????????????????}

\begin{document}

\begin{frontmatter}


\title{The Modified Camassa-Holm Equation: Wave breaking, Classification of Traveling Waves and Explicit Elliptic Peakons.}



\author{ALISSON DAR\'OS}
\address{Department of Mathematics, Federal University of Pampa,

97650-000 Itaqui, Brazil.

alissondaros@unipampa.edu.br}

\author{LYNNYNGS KELLY ARRUDA SARAIVA DE PAIVA}
\address{Department of Mathematics, Federal University of S\~ao Carlos,

PO B 676, 13565-905 S\~ao Carlos, Brazil.

lynnyngs@dm.ufscar.br}

{
\fontsize{10pt}{\baselineskip}\selectfont
\begin{abstract}
We show that wave breaking occurs for the modified Camassa-Holm (mCH) equation. Next we classify all traveling wave solutions of the modified Camassa-Holm equation in the weak sense via parametrization of their maxima, minima and wave velocity constants. This equation is shown to admit in addition to more popular solutions like smooth traveling waves and peakons, some not so well-known traveling waves as, for example, kinks, cuspons, composite waves and stumpons. Moreover, explicit peakons in terms of Jacobian elliptic functions are found.
\vspace{0.2cm}


\end{abstract}
}

%\begin{keyword}

%\end{keyword}

\end{frontmatter}

%\linenumbers

%\section{The First Section}
%\label{S:1}

\section{Introduction}
The original Camassa-Holm equation was introduced by Fuchssteiner and Fokas \cite{Fokas} through the method of recursion operators in 1981 and derived from physical principles by Camassa and Holm in 1999 \cite{Camassa}.

Our purpose is to investigate the modified Camassa-Holm equation (mCH)
			\begin{eqnarray}
				\label{mCH}
				u_{t}-u_{xxt} = uu_{xxx}+2u_{x}u_{xx}-3u^{2}u_{x},\ \ x\in\R,\ t>0\\
				u(x,0)=u_0(x),\label{dado}
				%u(x+L,t)=u(x,t), \ \ x\in\R,\ t>0, L>0.\label{percond}
				\end{eqnarray}
obtained from modified Dullin-Gottwald-Holm (mDGH) equation \cite{Yin}
					$$\tilde{u}_{t} + \kappa\tilde{u}_{x}   - \tilde{u}_{xxt} - \kappa \tilde{u}_{xxx} = \tilde{u}\tilde{u}_{xxx} + 2\tilde{u}_{x}\tilde{u}_{xx} - 3\tilde{u}^{2}\tilde{u}_{x},\ \ x\in\R,\ t>0,$$
by transformation $u(x,t)\mapsto \tilde{u}(x+\kappa t, t)$. The original DGH equation was obtained by Dullin, Gottwald and Holm \cite{Dullin} for a unidirectional water wave with fluid velocity $u(x,t)$, where the constant $\kappa\neq 0$ is the linear wave speed for undisturbed water at rest at spatial infinity. 

		The mCH equation can also be obtained from the $ab$-family of equations \cite{H1,Hakkaev,H}
				\begin{equation}
				\label{eq qntds conservadas}
				u_{t} + (a(u))_{x} - u_{xxt} = \left(b^{'}(u)\frac{u_{x}^{2}}{2} + b(u)u_{xx}\right)_{x}
				\end{equation}
where $a,b:\mathbb{R}\longrightarrow\mathbb{R}$ are smooth function and $a(0)=0$, by considering $a(u)=u^{3}$ and $b(u)=u$. From \cite{H1,H} it is also known that (\ref{mCH}) has three natural invariants:
				\begin{equation}
				\label{qntds conservadas}
				E(u)=-\int_{\mathbb R}\left[\frac{u^{4}}{8}+\frac{uu_{x}^{2}}{2}\right]\ dx, \ \ F(u)=\frac{1}{2}\int_{\mathbb R}[u^{2} +u_{x}^{2}]\ dx\ \ \ \text{and}\ \ \ V(u)= \int_{\mathbb R}u\ dx.
				\end{equation}
				%These conservation laws are the Hamiltonian and the Lagrangian which are given by the kinetic energy and the mean, respectively, for the evolution equation (\ref{mCH}).

		%To $u(t)\in H_{\text{per}}^{1}([0,L])$, with $t\geq 0$, we can write (\ref{mCH}) in the form
		%		\begin{equation}
		%			\label{estrutura hamiltoniana}
		%			u_{t} = J_{1}E^{'}(u),
		%			\end{equation}
%where $J_{1} = \partial_{x}(1-\partial_{x}^{2})^{-1}$ is a skew-symmetric linear Hamiltonian operator and $E^{'}$ denotes the derivative of G\^ateaux of functional $E$ in (\ref{qntds conservadas}), calculated in relation to the inner product of $L^{2}_{\text{per}}([0,L])$. In addition, the mCH equation has two Hamiltonian structure, because (\ref{mCH}) is also equivalent to
%											$$u_{t} = J_{2}F^{'}(u),$$
%with $J_{2} = -(1-\partial_{x}^{2})^{-1}[\partial_{x}(\frac{u^{2}}{2} - u_{xx}) + 2(u^{2} - u_{xx})\partial_{x}](1-\partial_{x}^{2})$ a linear Hamiltonian operator and $F(u)$ as in (\ref{qntds conservadas}). Despite having bi-hamiltonian structure there is no Lax formalism for (\ref{mCH}).

			Equation (\ref{mCH}) has been investigated in recent years. Tian and Song in \cite{Tian} obtained peakons composed of hyperbolic functions. In \cite{Wazwaz,Wazwaz2}, Wazwaz employed this modified form with $k=0$ as a vehicle to explore the change in the physical structure of the solution from peakons to bell-shaped solitary wave solutions and showed that mCH equation has a $sec\ h$ like solitary wave solution. In \cite{Yousif}, a variational homotopy perturbation method (VHPM) has been studied to obtain solitary wave solutions of the mCH equation. More recently in \cite{Daros}, we guarantee results about the orbital instability for a specific class of periodic traveling wave solutions with the mean zero property and large spatial period.

In this paper, we prove steepening at inflection points, i.e., we consider an initial condition in $H^3(\mathbb R)$ that has an inflection point to the right of its maximum and get  that the time dependent slope at the inflection point becomes vertical in finite time. This property means that wave breaking holds. Next we classify all traveling wave solutions, $u(x, t) = \phi(x -ct)$, $c \in \R$, for mCH equation $(\ref{mCH})$ using a weak formulation in $H^{1}_{loc}(\R)$ and obtain explicit formulas for peakons in terms of Jacobian elliptic functions. As for the mDP equation \cite{Ding}, we also obtain some very interesting types of solutions such as kinks, cuspons, composite waves and stumpons, in addition to the more familiar smooth waves and peakons. To get an idea: the composite waves are obtained by combining cuspons and peakons into new traveling waves and the waves called stumpons are obtained by inserting intervals where $\phi$ equals a constant at the crests of suitable cusped waves. 

%\begin{figure}[!ht]
%\centering{\resizebox{0.66\textwidth}{!}
%{\input{ondasdoteorema.pdf_tex}}}
%\caption{The traveling waves of (\ref{mCH}) corresponding to (a)-(i) of Theorem \ref{classificacao2}}\label{1}
%\end{figure}

%n Section \ref{preliminaries}, we present the key formulation to blowup phenomenon that will be described in Section \ref{blow}. Moreover, in Section \ref{well posed} we have the local well posedness of the Cauchy problem (\ref{mCH})-(\ref{dado}). In Section \ref{blow}, we presents the proof of the development of singularities in finite time for strong solutions with smooth initial data. 

%In section \ref{weak formulation}, we introduce the weak formulation of the modified Camassa-Holm equation that will be used in the classification theorems of Section \ref{classification}. Also, in Section \ref{classification}, we discuss the existence of some unbounded solutions to this equation. These results are demonstrated in Section \ref{proof classification}. Moreover, in Section \ref{regularity} we state the regularity of traveling waves of the mCH equation. Finally, in Section \ref{explicit peakons}, we determine explicit formulas for peaked traveling waves.

In the course of this work, we denote $\star$ the convolution on $\mathbb R$. We also use $(\cdot, \cdot)$ to represent the standard inner product in $L^2_{per}(\mathbb R)$. For $1\leq p \leq \infty$, the norm in the $L^p_{per}(\mathbb R)$ will be denoted by $||\cdot ||_{L^p}$, while $||\cdot||_{s}$ will stand for the norm in the classical Sobolev spaces $H^s_{per}(\mathbb R)$ for $s\geq 0$.

\section{Preliminaries}\label{preliminaries}

Formally, problem (\ref{mCH})-(\ref{dado}) is equivalent to the hyperbolic-elliptic system 
\begin{eqnarray}
u_t+\partial_x\left(\frac{u^2}{2}\right)+P_x =0,\ \  (x,t)\in \R \times \mathbb R_+, \label{H}\\
P-P_{xx}=u^3+\frac{u^2}{2}+\frac{{u_x}^2}{2}, \ \ (x,t)\in \R\times \mathbb R_+,\label{E}\\
u(x,0)=u_0(x),\ \ x\in \mathbb R\label{idata}
%u(x+L,t)=u(x,t), \ \ x\in\R,\ t>0, L>0.\label{percond1}
\end{eqnarray}

The operator $(1-\partial_{xx}^2)^{-1}$ has a convolution structure: 
 \begin{equation}
 (1-\partial_{xx}^2)^{-1}(f)(x)=(G\star f)(x),
 \end{equation}
 where  $G(x)$ is the Green function
 \begin{eqnarray*}
 G(x)=\frac{e^{-|x|}}{2}, \ \ x\in \mathbb R.
 \end{eqnarray*}
 %for $x\in \R$ and $[x]$ stands for the integer part of $x$. 
 Hence we have 
\begin{eqnarray}\label{P}
P(x,t)=G\star \left(u^3+\frac{u^2}{2}+\frac{u_x^2}{2}\right)(x,t),  
\end{eqnarray} 
and (\ref{H})-(\ref{idata}) can be rewritten as a conservation law with a non-local flux function $F$: 
\begin{eqnarray}
u_t+F(u)_x =0,\ \  (x,t)\in \mathbb R\times \mathbb R_+, \label{CL}\\
u(x,0)=u_0(x),\ \ x\in \mathbb R\label{idata1},
%u(x+L,t)=u(x,t), \ \ x\in\R,\ t>0, L>0.\label{percond2}
\end{eqnarray}		
where $F(u)=\frac{u^2}{2}+G\star (u^3+\frac{u^2}{2}+\frac{u_x^2}{2})$.
%Next  we present a Sobolev-type inequality which plays a key role to obtain blow-up results for the initial-value problem (\ref{mCH})-(\ref{percond}) in the sequel. Then based on the first three  conservation laws, we will prove some a priori estimates.

%\begin{lemma}\label{average}
%{\rm \cite{B}} For every $f(x)\in H^1(a,b)$ periodic and with zero average, i.e., such that $\int_a^b f(x)dx=0$, we have 
%$$ \int_a^b f^2(x)dx \leq \left ( \frac{b-a}{2\pi}\right)^2\int_a^b|f'(x)|^2dx, $$
%and equality holds if and only if 
%$$ f(x)=A\cos \left( \frac{2\pi x}{b-a}\right)+B\sin \left( \frac{2\pi x}{b-a}\right).$$
%\end{lemma}	
\begin{definition}\label{ws} Let $u_0\in H^1(\mathbb R)$ be given. A function $u:[0,T]\times \mathbb{R} \rightarrow \mathbb{R} $ is called a weak solution to (\ref{mCH}), if $u\in L^{loc}_{\infty}([0,T]; H^1)$ satisfies the identity
$$ \int_0^T \int_{\mathbb R} (u\psi_t+F(u)\psi_x)dxdt+\int_{\mathbb R}u_0(x)\psi_0(x)dx=0$$
for all $\psi\in C^{\infty}_0([0,T]\times \mathbb{R})$ that are restrictions to $[0,T)\times \mathbb{R}$
of a continuously differentiable function on $\mathbb{R}^2$ with compact support contained in $(-T,T)\times \mathbb{R}.$
\end{definition}
		
\section{Steepening lemma}
The local well posedness of the Cauchy problem of Equation (\ref{mCH}) with initial data $u_0\in H^s(\mathbb R)$,  $s>\frac{3}{2}$ can be obtained by applying Kato’s semigroup theory \cite{Kato}. More precisely, we have the following well-posedness result. 

\begin{theorem}
\label{teorema local}
Given $u_{0}\in H^{s}(\mathbb R)$, $s>\frac{3}{2}$, there exists a maximal $t_0> 0$ and a unique solution $u(x,t)$ to problem (\ref{mCH})-(\ref{dado}) such that
				$$u\in C([0,t_{0}), H^{s}(\mathbb R))\cap  C^{1}([0,t_{0}), H^{s-1}(\mathbb R)).$$
Moreover, the solution depends continuously on the initial data. For $u_{0}\in H^{3}(\mathbb R)$ the solution posseses the aditional regularity 
$$u\in C([0,t_{0}), H^{3}(\mathbb R))\cap  C^{1}([0,t_{0}), H^{2}(\mathbb R)). $$
  
\end{theorem}
\begin{proof}
See Hakkaev, Iliev and Kirchev in \cite{H}.
\end{proof}
\begin{remark} The solutions obtained in Theorem \ref{teorema local} are called strong solutions to Equation (\ref{mCH}).
\end{remark}
%\begin{definition}
%If $$u\in C([0,t_{0}), H^{s}(\mathbb R))\cap  C^{1}([0,t_{0}), H^{s-1}(\mathbb R)),$$
%with $s>\frac{3}{2}$ is a solution to (\ref{mCH}), the $u(x,t)$ is called a strong solution to (\ref{mCH}).
%\end{definition}

Assume now that $u_0\in H^3(\mathbb{R})$ and let $u\in C([0,t_0);H^3(\mathbb{R}) )\cap C^1([0,t_0);H^2(\mathbb{R}) )$ be the corresponding strong solution of (\ref{mCH})-(\ref{dado}). 

\begin{remark}
If $u_0\in H^3(\mathbb{R})$, then the $|| u(\cdot,t)||_1$ norm is conserved in time as long as the solution exists, which implies by the Sobolev embedding $H^1(\mathbb R)\subset L_{\infty}(\mathbb R)$ that  
\end{remark}
\begin{equation}\label{tbound}
M:=\sup_{t\in [0,\infty)} ||u(\cdot,t) ||_{L_{\infty}} <\infty.
\end{equation}

%\section{Blowup results}
%\label{blow}

%We will prove that there are smooth initial data for which the corresponding solution of (\ref{mCH})-(\ref{percond}) does not exist globally. 

%Let us first derive a useful identity satisfied by a solution to (\ref{mCH})-(\ref{percond}).

The operator $(1-\partial_{xx}^2)^{-1}$ can be represented as a convolution operator:
$$((1-\partial_{xx}^2)^{-1} )f(x)=\int_{\mathbb R} G(x-y)f(y)dy,\ \ f\in L^2(\mathbb R),$$ where 
$$  G(x)=\frac{e^{-|x|}}{2},\ \ x\in\R.$$
From equation (\ref{CL}) we get
\begin{eqnarray}\label{advective}
u_t+uu_x&=&-\partial_x\left(G\star \left(u^3+\frac{u^2}{2}+\frac{u_x^2}{2}\right)\right)\nonumber\\
&=&-\partial_x\int_{\mathbb R}\frac{1}{2}\exp(-|x-y|)\left(u^3(y,t)+\frac{u^2}{2}(y,t)+\frac{u_y^2}{2}(y,t)  \right)dy
\end{eqnarray}
 in $C([0,T);H^1(\mathbb R) )$.
%where $\star $ stands for the convolution with respect to the spatial variable. 

By differentiation with respect to $x$ we get 
\begin{eqnarray*}
u_{tx}+u_x^2+uu_{xx}&=&-\partial_x^2\left(G\star \left(u^3+\frac{u^2}{2}+\frac{u_x^2}{2}\right)\right)\\
&=& (Q^2-Id)\left(G\star \left(u^3+\frac{u^2}{2}+\frac{u_x^2}{2}\right)\right)\\
&=& \left(u^3+\frac{u^2}{2}+\frac{u_x^2}{2}\right)-\left(G\star \left(u^3+\frac{u^2}{2}+\frac{u_x^2}{2}\right)\right),
\end{eqnarray*}
and therefore 
\begin{equation}\label{eb}
 u_{tx}+uu_{xx}=u^3+\frac{u^2}{2}-\frac{u_x^2}{2}-\left(G\star \left(u^3+\frac{u^2}{2}+\frac{u_x^2}{2}\right)\right)
\end{equation}
in the space $C([0,T);L^2(\mathbb R) ).$

%The mechanism for the formation of the peakons seen may also be understood as a variant of classical formations of weak solutions in fluid dynamics by showing that initial conditions exist for which the solution of the mCH equation (\ref{mCH}) can develop a vertical slope in its velocity $u(t,x)$, in finite time. 

We prove now the following blow-up result for (\ref{mCH})-(\ref{dado}):

\begin{lemma}(Steepening Lemma \cite{Camassa,Crisan}) Suppose the initial profile of velocity $u_0\in H^{3}(\mathbb R)$, has an inflection point at $x=\overline{x}$ to the right of its maximum. Moreover we assume that $u_x(\overline x,0)<-\sqrt{2(M^2+2M^3)}$ , where $M $ is the constant defined in (\ref{tbound}).Then, the negative slope at the inflection point will become vertical in finite time.
\end{lemma}
\begin{proof}
Consider the evolution of the slope at the inflection point $t \rightarrow \overline{x}(t)$ that starts at time $0$ from an inflection point $x=\overline{x}$ of $u_0(x)$ to the right of its maximum so that
$$ \rho_0=u_x(\overline{x}(0),0)< \infty.$$

Define $\rho_t=u_x(\overline{x}(t),t), \ \ t\geq 0$. 

Equation (\ref{eb}) yields an equation for the evolution of $t\rightarrow \rho_t$. Namely, by using that $u_{xx}(\overline{x}(t),t)=0$ $(u(t)\in H^3(\mathbb R)\subset C^2(\mathbb R))$ and (\ref{tbound}) one finds 

\begin{eqnarray}
\frac{d{\rho}_t}{dt} &=& -\frac{{\rho}_t^2}{2}+u^3(\overline{x}(t),t)+\frac{u^2}{2}(\overline{x}(t),t)\nonumber\\ &-&\int_{\mathbb R}\frac{1}{2}\exp(-|\overline{x}(t)-y|)\left(u^3(y,t)+\frac{u^2}{2}(y,t)+\frac{u_y^2}{2}(y,t)  \right)dy\nonumber\\ 
&\leq &  -\frac{{\rho}_t^2}{2} +M^3+ M^2+ \int_{\mathbb R}\frac{1}{2}\exp(-|\overline{x}(t)-y|)(u_{-}^3(y,t))dy\nonumber\\ 
& \leq &-\frac{{\rho}_t^2}{2} + M^2+2M^3,
\end{eqnarray}
where $u_-$ stands for the negative part of $u$. Here, we used that $ \int_{\mathbb R}\frac{1}{2}\exp(-|\overline{x}(t)-y|)dy=1.  $

Let $\tilde{\rho}_t$ be the solution of the equation 
\begin{equation}\label{edo}
\frac{d{\tilde{\rho}_t}}{dt}=-\frac{\tilde{\rho}_t^2}{2} + \overline{M}, \ \ \tilde{\rho}_0=\rho_0,
\end{equation}
where $\overline{M}= M^2+2M^3.$

Observe that 
\begin{eqnarray*}
\frac{d}{dt}\left((\rho_t-\tilde{\rho}_t)e^{\frac{1}{2}\int_0^t(\rho_s+\tilde{\rho}_s)ds}\right)\leq 0, \ \ \rho_0-\tilde{\rho}_0=0,
\end{eqnarray*}
therefore $\rho_t\leq \tilde{\rho}_t$ for all $t>0$ for which both are well defined.  

Integrating (\ref{edo}) we obtain
\begin{eqnarray}
\coth^{-1}\left(\frac{\tilde{\rho}_t}{\sqrt{2\overline{M}}}\right)=\frac{1}{2} \ln \left| \frac{\tilde{\rho}_t+\sqrt{2\overline{M}}}{\tilde{\rho}_t-\sqrt{2\overline{M}}}\right|=\sqrt{2\overline{M}}\frac{t}{2}+k,
\end{eqnarray} 
where $k= \ln \left| \frac{{\rho}_0+\sqrt{2\overline{M}}}{{\rho}_0-\sqrt{2\overline{M}}}\right|<0$.  Since $ \lim_{t\rightarrow -2k/ \sqrt{2\overline{M}}} \tilde{\rho}_t=-\infty$ it follows that there is a time $\tau\leq -2k/ \sqrt{2\overline{M}}$  by which the slope $\rho_t = u_x(\overline{x}(t),t)$ becomes negative and vertical.
\end{proof}

\section{Weak formulation}\label{weak formulation}

		For a traveling wave $u(x,t)=\phi(\xi)$, with $\xi=x-ct$, the equation (\ref{mCH}) is equivalent to
			\begin{equation}
			\label{169}
			-c\phi^{'} +c\phi^{'''} = \phi\phi^{'''} + 2\phi^{'}\phi^{''} - 3\phi^{2}\phi^{'},
			\end{equation}
where $'$ denotes the derivative with respect to the variable $\xi$. 
			
%			Since $(\phi\phi^{''})^{'} = \phi^{'}\phi^{''} + \phi\phi^{'''}$, $[\frac{(b-1)}{2}(\phi^{'})^{2}]^{'} = (b-1)\phi^{'}\phi^{''}$ and $[\frac{(b+1)}{3}\phi^{3}]^{'} = (b+1)\phi^{2}\phi^{'}$, after an integration with respect to the variable $\xi$ we can rewrite (\ref{169}) as follows
%			\begin{equation}
%			\label{170}
%			-c\phi + c\phi^{''} + \frac{(b+1)}{3}\phi^{3} + bk\phi = \phi\phi^{''} + \frac{(b-1)}{2}(\phi^{'})^{2} + \frac{a}{2},
%			\end{equation}
%for some constant $a\in\R$. Equivalently, we have from (\ref{170})

			After an integration with respect to the variable $\xi$ we can rewrite (\ref{169}) as follows
			\begin{equation}
			\label{171}
			(\phi^{'})^{2} + 2\phi^{3} -2c\phi = ((\phi -c)^{2})^{''} +a.
			\end{equation}
for some constant of integration $a\in\R$. We see that for (\ref{171}) to make any sense it is sufficient to require $\phi\in H^{1}_{loc}(\R)$. Thus, the following definition is plausible.

\begin{definition}
\label{205}
A function $\phi\in H^{1}_{loc}(\R)$ is a traveling wave of the mCH equation if $\phi$ satisfies (\ref{171}) in distribution sense for some $a\in\R$.
\end{definition}

\section{Classification of traveling waves}\label{classification}

Theorems \ref{classificacao1} and \ref{classificacao2} classify all bounded traveling waves $\phi\in H^{1}_{loc}(\R)$, of (\ref{mCH}). The bounded traveling waves are parametrized by their maxima, minima and speeds of propagation.

\begin{theorem}
\label{classificacao1}
Consider $z\in\C\setminus\R$ such that $\Re(z)=-(m+M)/2$ satisfying the equation $m^{2} + M^{2} -|z|^{2} + mM -2c = 0$. Any bounded traveling wave of the mCH equation belongs to one of the following categories:

$(a)$ Smooth periodic: If $m<M<c$, there is a smooth periodic traveling wave $\phi(x-ct)$ of (\ref{mCH}) with $m=\min_{\xi\in\R}\phi(\xi)$ and $M=\max_{\xi\in\R}\phi(\xi)$.

$(b)$ Periodic peakons: If $m<M=c$, there is a periodic peaked traveling wave $\phi(x-ct)$ of (\ref{mCH}) with $m=\min_{\xi\in\R}\phi(\xi)$ and $M=\max_{\xi\in\R}\phi(\xi)$.

$(c)$ Periodic cuspons: If $m<c<M $, there is a periodic cusped traveling wave $\phi(x-ct)$ of (\ref{mCH}) with $m=\min_{\xi\in\R}\phi(\xi)$ and $c=\max_{\xi\in\R}\phi(\xi)$.

$(a^{'})$ Smooth periodic: If $c<M<m$, there is a smooth periodic traveling wave $\phi(x-ct)$ of (\ref{mCH}) with $M=\min_{\xi\in\R}\phi(\xi)$ and $m=\max_{\xi\in\R}\phi(\xi)$.

$(b^{'})$ Periodic peakons: If $c=M<m$, there is a periodic peaked traveling wave $\phi(x-ct)$ of (\ref{mCH}) with $M=\min_{\xi\in\R}\phi(\xi)$ and $m=\max_{\xi\in\R}\phi(\xi)$.

$(c^{'})$ Periodic cuspons: If $M<c<m$, there is a periodic cusped traveling wave $\phi(x-ct)$ of (\ref{mCH}) with $c=\min_{\xi\in\R}\phi(\xi)$ and $m=\max_{\xi\in\R}\phi(\xi)$.

$(d)$ Composite waves: For each $c\in \R$ fixed, with $c <0$, consider the equation
									\begin{equation}
									\label{303}
									2a = (M+m)(|z|^{2} -mM)
									\end{equation}
in space $(m,M,\Im(z))$ with $z\in\C\setminus\R$ satisfying, less than a change of variables, the hyperboloid of two sheet 
					\begin{equation}
					\label{304}
					\frac{m^{2}}{\alpha^{2}} - \frac{M^{2}}{\beta^{2}} - \frac{\Im(z)^{2}}{\gamma^{2}} = 1,
					\end{equation}
where $\alpha = \sqrt{-2c}$, $\beta = 2\sqrt{-c}$ and $\gamma = \sqrt{-c}$. For any $(m,M,c)$ $\in \{m <c \leqslant M\}\cup\{m >c \geqslant M\}$ there is a corresponding cuspon or peakon according to $(b)$ or $(c)$ and $(b^{'})$ or $(c^{'})$. A countable number of cuspons and peakons corresponding to points $(m,M,c)$ that belong to the same hyperboloid of two sheet, may be joined at their crests to form a composite wave $\phi$. If the Lebesgue measure $\mu(\phi^{-1}(c)) = 0$, then $\phi$ is a traveling wave of (\ref{mCH}).

$(e)$ Stumpons: For $a=2c^{3} -2c^{2}$ the equation (\ref{303}) jointly with (\ref{304}) contains the points $(m,M,\Im(z)) = (c,c, \pm\sqrt{4c^{2}-2c})$. These equations correspond only to cuspons. Let $\phi$ be a composite wave obtained by joining countably many of these cuspons with each other and with intervals where $\phi\equiv c$. Then $\phi$ is a traveling wave of (\ref{mCH}) even if $\mu(\phi^{-1}(c)) >0$. 
\end{theorem}

\begin{theorem}
\label{classificacao2}
Consider $z= -m-r-M$ with $r\in\R$ satisfying the equation $r^{2} + m^{2} + M^{2} +rm +rM + mM -2c = 0$. Any bounded traveling wave of the mCH equation belongs, less than an alternation of $z$ for $r$, to one of the following categories:

$(a)$ Smooth periodic: If $z<r<m<M<c$, there is a smooth periodic traveling wave $\phi(x-ct)$ of (\ref{mCH}) with $m=\min_{\xi\in\R}\phi(\xi)$ and $M=\max_{\xi\in\R}\phi(\xi)$.

$(b)$ Smooth with decay: If $z<r=m<M<c$, there is a smooth traveling wave $\phi(x-ct)$ of (\ref{mCH}) with $m=\inf_{\xi\in\R}\phi(\xi)$, $M=\max_{\xi\in\R}\phi(\xi)$ and $\phi\downarrow m$ exponentially as $\xi\rightarrow\pm\infty$.

$(c)$ Kinks: If $z=m<M=r<c$, there is a kinked smooth traveling wave $\phi(x-ct)$ of (\ref{mCH}) with $m=\inf_{\xi\in\R}\phi(\xi)$, $M=\sup_{\xi\in\R}\phi(\xi)$ and, $\phi\uparrow M$ and $\phi\downarrow m$ exponentially as $\xi\rightarrow\pm\infty$, respectively.

$(d)$ Periodic peakons: If $z<r<m<M=c$, there is a periodic peaked traveling wave $\phi(x-ct)$ of (\ref{mCH}) with $m=\min_{\xi\in\R}\phi(\xi)$ and $M=\max_{\xi\in\R}\phi(\xi)$.

$(e)$ Peakons with decay: If $z<r=m<M=c$, there is a peaked traveling wave $\phi(x-ct)$ of (\ref{mCH}) with $m=\inf_{\xi\in\R}\phi(\xi)$, $M=\max_{\xi\in\R}\phi(\xi)$ and $\phi\downarrow m$ exponentially as $\xi\rightarrow\pm\infty$.

$(f)$ Periodic cuspons: If $z<r<m<c<M $, there is a periodic cusped traveling wave $\phi(x-ct)$ of (\ref{mCH}) with $m=\min_{\xi\in\R}\phi(\xi)$ and $c=\max_{\xi\in\R}\phi(\xi)$.

$(g)$ Cuspons with decay: If $z<r=m<c<M$, there is a cusped traveling wave $\phi(x-ct)$ of (\ref{mCH}) with $m=\inf_{\xi\in\R}\phi(\xi)$, $c=\max_{\xi\in\R}\phi(\xi)$ and $\phi\downarrow m$ exponentially as $\xi\rightarrow\pm\infty$.

$(a^{'})$ Smooth periodic: If $c<M<m<r<z$, there is a smooth periodic traveling wave $\phi(x-ct)$ of (\ref{mCH}) with $M=\min_{\xi\in\R}\phi(\xi)$ and $m=\max_{\xi\in\R}\phi(\xi)$.

$(b^{'})$ Smooth with decay: If $c<M<m=r<z$, there is a smooth traveling wave $\phi(x-ct)$ of (\ref{mCH}) with $M=\min_{\xi\in\R}\phi(\xi)$, $m=\sup_{\xi\in\R}\phi(\xi)$ and $\phi\uparrow m$ exponentially as $\xi\rightarrow\pm\infty$.

$(c^{'})$ Kinks: If $c<r=M<m=z$, there is a kinked smooth traveling wave $\phi(x-ct)$ of (\ref{mCH}) with $M=\inf_{\xi\in\R}\phi(\xi)$, $m=\sup_{\xi\in\R}\phi(\xi)$ and, $\phi\downarrow M$ and $\phi\uparrow m$ exponentially as $\xi\rightarrow\pm\infty$, respectively.

$(d^{'})$ Periodic peakons: If $c=M<m<r<z$, there is a periodic peaked traveling wave $\phi(x-ct)$ of (\ref{mCH}) with $M=\min_{\xi\in\R}\phi(\xi)$ and $m=\max_{\xi\in\R}\phi(\xi)$.

$(e^{'})$ Peakons with decay: If $c=M<m=r<z$, there is a peaked traveling wave $\phi(x-ct)$ of (\ref{mCH}) with $M=\min_{\xi\in\R}\phi(\xi)$, $m=\sup_{\xi\in\R}\phi(\xi)$ and $\phi\uparrow m$ exponentially as $\xi\rightarrow\pm\infty$.

$(f^{'})$ Periodic cuspons: If $M<c<m<r<z$, there is a periodic cusped traveling wave $\phi(x-ct)$ of (\ref{mCH}) with $c=\min_{\xi\in\R}\phi(\xi)$ and $m=\max_{\xi\in\R}\phi(\xi)$.

$(g^{'})$ Cuspons with decay: If $M<c<m=r<z$, there is a cusped traveling wave $\phi(x-ct)$ of (\ref{mCH}) with $c=\min_{\xi\in\R}\phi(\xi)$, $m=\sup_{\xi\in\R}\phi(\xi)$ and $\phi\uparrow m$ exponentially as $\xi\rightarrow\pm\infty$.

$(h)$ Composite waves: For each $c\in \R$ fixed, with $c >0$, the equation
									\begin{equation}
									\label{203}
									2a = -(M+m)r^{2} - (m+r)M^{2} - (M+r)m^{2} -2mrM
									\end{equation}
describes three planes intersecting in space $(m,M,r)$ with $r\in\R$ satisfying, less than a change of variables, the ellipsoid
					\begin{equation}
					\label{204}
					\frac{m^{2}}{\alpha^{2}} + \frac{M^{2}}{\beta^{2}} + \frac{r^{2}}{\gamma^{2}} = 1,
					\end{equation}
where $\alpha =\beta = 2\sqrt{c}$ and $\gamma = \sqrt{c}$. For any $(m,M,r,c)\in \{z<r\leqslant m <c \leqslant M\}\cup\{z>r\geqslant m >c \geqslant M\}$ there is a corresponding cuspon or peakon according to $(d)-(g)$ and $(d^{'})-(g^{'})$. A countable number of cuspons and peakons corresponding to points $(m,M,r,c)$ that belong to the same ellipsoid, may be joined at their crests to form a composite wave $\phi$. If the Lebesgue measure $\mu(\phi^{-1}(c)) = 0$, then $\phi$ is a traveling wave of (\ref{mCH}).

$(i)$ Stumpons: For $a=2c^{3} -2c^{2}$ the equation (\ref{203}) jointly with (\ref{204}) contains the points $(m,M,r) = (c,c,-c\pm\sqrt{-2c^{2}+2c})$. These equations correspond only to cuspons. Let $\phi$ be a composite wave obtained by joining countably many of these cuspons with each other and with intervals where $\phi\equiv c$. Then $\phi$ is a traveling wave of (\ref{mCH}) even if $\mu(\phi^{-1}(c)) >0$. 
\end{theorem}

			In addition to Theorems \ref{classificacao1} and \ref{classificacao2} is possible to characterize the existence or no of some unbounded traveling waves of the mCH equation.
			
\begin{theorem}
\label{classificacao3}
If the polynomial $P(\phi)$ in (\ref{188}) has only a real zero, say $m$, then the mCH equation has not bounded solutions $\phi$ to $c<m$ or $m<c$ with $m=\min_{\xi\in\R}\phi(\xi)$ or $m=\max_{\xi\in\R}\phi(\xi)$, respectively. In particular, the mCH equation has traveling wave solutions with exponential behavior.
\end{theorem}
	
\begin{remark}
Differently from the classical case, the mCH equation has not traveling wave solutions with similar behavior to a parabola because this is only possible if the polynomial $P(\phi)$ in (\ref{188}) has only a real zero with multiplicity equal to one.
\end{remark}

%\begin{figure}[!ht]
%\centering{\resizebox{0.35\textwidth}{!}
%{\input{unboundedsolutions.pdf_tex}}}
%\caption{Unbounded traveling waves}\label{2}
%\end{figure}

\begin{remark}
			For the modified Degasperis-Procesi equation
							$$u_{t}-u_{xxt} = uu_{xxx} +3u_{x}u_{xx} - 4u^{2}u_{x},$$
we can also classify all traveling wave solutions for this equation and get similar results like the above theorems. In order to get this, we proceed the same way as for the mCH equation, by considering in the next section $p(v)=\frac{8}{3}\phi^{3} -2c\phi - a$ and Lemma 1 of \cite{Lenells2} in the proof of Lemma \ref{lema4} in this paper. Moreover, to obtain a equation like $(\ref{180})$ we multiply by $\phi\phi^{'}$ the weak formulation to mDP equation given by
								$$\frac{8}{3}\phi^{3} -2c\phi = ((\phi-c)^{2})^{''} + a,$$
where $a$ is a constant of integration as in $(\ref{171})$.
\end{remark}

\section{Proof of Theorems \ref{classificacao1}, \ref{classificacao2} and \ref{classificacao3}}\label{proof classification}

\begin{lemma}
\label{lema1}
Let $p(v)$ be a polynomial with real coefficients. Assume that $v\in H^{1}_{loc}(\R)$ satisfies 
			$$(v^{2})^{''} = (v^{'})^{2} + p(v)\ \ \text{in}\ \ \mathcal{D}^{'}(\R).$$
Then, $$v^{k}\in C^{j}(\R)\ \ \text{para}\ \ k\geqslant 2^{j}.$$
\end{lemma}
\begin{proof}
See \cite{Lenells}, Lemma 1, page 404.
\end{proof}

\begin{lemma}
\label{lema2}
Let $f:\R\longrightarrow\R$ be an absolutely continuous function. Then $f^{'}=0$ a.e. on $f^{-1}(c)$ for any $c\in\R$.
\end{lemma}
\begin{proof}
See \cite{Lenells}, Lemma 2, page 405.
\end{proof}

\begin{lemma}
\label{lema3}
Let $f\in W^{2,1}_{loc}(\R)$. Then, $f^{''}=0$ a.e. on $f^{-1}(c)$ for any $c\in\R$.
\end{lemma}
\begin{proof}
See \cite{Lenells}, Lemma 3, page 405.
\end{proof}

\begin{lemma}
\label{lema4}
A function $\phi\in H^{1}_{loc}(\R)$ is a traveling wave of (\ref{mCH}) with speed $c$ if and only if
the following three statements hold:

\noindent\noindent $(TW1)$ There are disjoint open intervals $E_{i}$, $i\geqslant 1$, and a closed set $C$ such that $\R\setminus C = \bigcup^{+\infty}_{i=1}E_{i}$, $\phi\in C^{\infty}(E_{i})$ for $i\geqslant 1$, $\phi(\xi)\neq c$ for $\xi\in E\ \dot{=}\ \bigcup_{i=1}^{+\infty}E_{i}$, and $\phi(\xi)=c$ for $\xi\in C$.

\noindent\noindent $(TW2)$ There is an $a\in\R$ such that

\ $(i)$ For each $i\geqslant 1$, there exists $d_{i}\in\R$ such that
				\begin{equation}
				\label{178}
				(\phi^{'})^{2} = F(\phi)\ \text{for}\ \xi\in E_{i}\ \ \text{and}\ \ \phi\rightarrow c\ \text{at any finite endpoint of}\ E_{i},
				\end{equation}
where, 
				\begin{equation}
				\label{179}
				F(\phi) = \frac{\phi^{2}\left(c-\frac{\phi^{2}}{2}\right) +a\phi +d_{i}}{c-\phi};
				\end{equation}
				
\ $(ii)$ If $C$ has strictly positive Lebesgue measure, $\mu(C) >0$, then $a=2c^{3} -2c^{2}$.

\noindent\noindent $(TW3)$ $(\phi -c)^{2}\in W^{2,1}_{loc}(\R)$.
\end{lemma}
\begin{proof}
				Applying Lemma \ref{lema1}, with $v=\phi -c$ and $p(v) = 2\phi^{3} -2c\phi -a$, we see that 
								$$(\phi -c)^{k}\in C^{j}(\R)\ \ \text{para}\ \ k\geqslant 2^{j}.$$

			Define $C=\phi^{-1}(c)$. Since $\phi$ is continuous, $C$ is a closed set. Since every open set is a countable union of disjoint open intervals, there are disjoint open intervals $E_{i}$, $i\geqslant 1$, such that $\R\setminus C = \bigcup_{i=1}^{\infty}E_{i}$. By construction, it follows that $(TW1)$ is satisfied.
			
			Now let's prove $(TW2)$. To get this, we consider $E_{i}$ one of these open intervals. Since $\phi$ is $C^{\infty}$ in $E_{i}$, we infer that (\ref{171}) holds pointwise in $E_{i}$. So, multiplying (\ref{171}) by $\phi^{'}$ we obtain
			$$-2c\phi\phi^{'} +2c\phi^{'}\phi^{''} +2\phi^{3}\phi^{'} = 2\phi\phi^{'}\phi^{''} + (\phi^{'})^{3} + a\phi^{'}.$$
%So, since $(\phi^{2})^{'} = 2\phi\phi^{'}$, $[(\phi^{'})^{2}]^{'}= 2\phi^{'}\phi^{''}$, $(\frac{\phi^{4}}{2})^{'} = 2\phi^{3}\phi^{'}$ and $[(\phi^{'})^{2}\phi]^{'} = 2\phi\phi^{'}\phi^{''} + (\phi^{'})^{3}$, the equation above can be rewritten, equivalently, by
Equivalently, we can rewrite the above equation as
						$$-c(\phi^{2})^{'} + c[(\phi^{'})^{2}]^{'} + \left(\frac{\phi^{4}}{2}\right)^{'} = [(\phi^{'})^{2}\phi]^{'} +a\phi^{'}$$
and from an integration with respect to the parameter $\xi$, we deduce
				\begin{equation}
				\label{180}
				(\phi^{'})^{2}(c-\phi) = \phi^{2}\left(c-\frac{\phi^{2}}{2}\right) +a\phi +d_{i},\ \xi\in E_{i}
				\end{equation}
for some constants of integration $d_{i}$. Dividing (\ref{180}) by $c-\phi$ we have $(\ref{178})$. That $\phi\rightarrow c$ at the finite
endpoints of $E_{i}$ it follows from the continuity of $\phi$ and $(TW1)$. This proves $(i)$ of $(TW2)$.

				The left-hand side of $(\ref{171})$ is in $L^{1}_{loc}(\R)$. Hence, $((\phi -c)^{2})^{''}\in L^{1}_{loc}(\R)$ and so follows $(TW3)$.
				
				To show $(ii)$ of $(TW2)$, let us assume $\mu(C)>0$. Since $\phi\in H^{1}_{loc}(\R)$ and $(\phi -c)^{2}\in W^{2,1}_{loc}(\R)$, we obtain from Lemma \ref{lema2} and Lemma \ref{lema3} respectively,
				\begin{equation}
				\label{181}
				\phi^{'}(\xi) = 0\ \ \text{and}\ \ ((\phi -c)^{2})^{''}(\xi) = 0\ \text{a.e. in}\ C.
				\end{equation}
By the fact that $(\phi -c)^{2}\in W^{2,1}_{loc}(\R)$, we have that (\ref{171}) holds a.e. in $\R$, i.e., 
				$$(\phi^{'})^{2} + 2\phi^{3} -2c\phi = ((\phi -c)^{2})^{''} +a\ \ \text{a.e. in}\ \R .$$
In particular, the above equation occurs a.e. in $C$. Then, from (\ref{181}) it follows that
				$$2\phi^{3} -2c\phi = a\ \ \text{a.e. in}\ C.$$
Since $\mu(C)>0$ and $\phi\equiv c$ in $C$, we conclude that $a=2c^{3} -2c^{2}$. This shows that all the traveling waves of equation (\ref{mCH}) satisfy $(TW1)-(TW3)$.

		Reciprocally, note that from the differentiation of (\ref{178}) we have
					\begin{equation}
					\label{182}
					(\phi^{'})^{2} + 2\phi^{3} -2c\phi = ((\phi -c)^{2})^{''} +a\ \ \text{in}\ E\ \dot{=}\ \bigcup_{i=1}^{\infty} E_{i}.
					\end{equation}
If $\mu(C) = 0$, then (\ref{182}) implies that
					\begin{equation}
					\label{183}
					(\phi^{'})^{2} + 2\phi^{3} -2c\phi = ((\phi -c)^{2})^{''} +a\ \ \text{a.e. in}\ \R .
					\end{equation}
Since $((\phi -c)^{2})^{''}\in L^{1}_{loc}(\R)$ by $(TW3)$, (\ref{183}) implies (\ref{171}), from which we conclude that $\phi$ is a traveling wave solution of (\ref{mCH}).
		
			To complete the demonstration it remains to show that (\ref{183}) also occurs in the case where $\mu(C)>0$. Suppose $\mu(C)>0$. Since $\phi\in H^{1}_{loc}(\R)$ and $(\phi -c)^{2}\in W^{2,1}_{loc}(\R)$, we obtain from Lemmas \ref{lema2} and \ref{lema3} that equation (\ref{181}) holds. From $(ii)$ of $(TW2)$ we have $a=2c^{3}-2c^{2}$. So, since $\phi\equiv c$ in $C$, we deduce for a.e. $\xi$ in $C$,
			%$$(\phi^{'})^{2} + 2c^{3}-2c^{2}+4kc = ((\phi -c)^{2})^{''} +a \Leftrightarrow (\phi^{'})^{2} + 2\phi^{3} + (4k-2c)\phi = ((\phi -c)^{2})^{''} +a.$$
			$$(\phi^{'})^{2} + 2\phi^{3} -2c\phi = ((\phi -c)^{2})^{''} +a.$$
Jointly with (\ref{182}), what we have just shown implies (\ref{183}) and the result follows.
\end{proof}
						
			To prove Theorems \ref{classificacao1} and \ref{classificacao2}, we will show that the set of bounded functions satisfying $(TW1)-(TW3)$ in Lemma \ref{lema4} consists exactly of the waves presented in the statements of these theorems.
						
			Suppose that $\phi$ satisfies $(TW1)-(TW2)$. From $(TW1)$ and $(TW2)$ there is a countable number of $C^{\infty}$ wave segments separated by a closed set $C$ such that each wave segment $\phi$ satisfies
				\begin{eqnarray}
				\label{184}
				&& (\phi^{'})^{2} = F(\phi)\ \text{for}\ \xi\in E,\ \ F(\phi) = \frac{\phi^{2}\left(c-\frac{\phi^{2}}{2}\right) +a\phi +d}{c-\phi}\nonumber\\
				&&\text{and}\ \ \phi\rightarrow c\ \text{at any finite endpoint of}\ E
				\end{eqnarray}
for some interval $E$ and constants $a,d$. If we are able to find all solutions of $(\ref{184})$ for different intervals $E$ and different values of $a$ and $d$, then we can join solutions defined at intervals in which the union is $\R\setminus C$ for some closed set $C$ of null measure. We will have the function that we constructed defined in $\R$, satisfying $ (TW1) $ and $ (TW2) $ if, and only if, all the wave segments satisfy (\ref{184}) with the same constant $a$. In addition, if for $a=2c^{3}-2c^{2}$ we allow $\mu(C)>0$, this procedure will give us all the functions satisfying $(TW1)$ and $(TW2)$. Let us show that these functions we have just constructed belong to $H^{1}_{loc}(\R)$, satisfy $(TW3)$ and are exactly the waves in Theorems \ref{classificacao1} and \ref{classificacao2}.

To study (\ref{184}), we first observe that for general equations of the form
				\begin{equation}
				\label{185}
				(\phi^{'})^{2} = F(\phi),
				\end{equation}
where $F:\R\longrightarrow\R$ is a rational function, Lenells in \cite{Lenells} explored the qualitative behavior of solutions to (\ref{185}) at points where $F$ has a zero or a pole. In resume, if $F(\phi)$ has a simple zero at $\phi = m$, so that $F^{'}(m) >0$, then the solution $\phi$ of (\ref{185}) satisfies
			\begin{equation}
			\label{eq polinomial}
			\phi(\xi) = m+\frac{1}{4}(\xi-\eta)^{2}F^{'}(m) + O((\xi-\eta)^{4})\ \ \text{at}\ \ \xi\uparrow\eta,
			\end{equation}
where $\phi(\eta)=m$. Assuming that $F(\phi)$ has a double zero at $m$, so that $F^{'}(m) = 0$ and $F^{''}(m)>0$, we get
			\begin{equation}
			\label{eq exp}
			\phi(\xi) - m \approx \alpha\textit{exp}(-\xi\sqrt{F^{''}(m)})\ \ \text{as}\ \ \xi\rightarrow \infty
			\end{equation}
for some constant $\alpha$, then $\phi \downarrow m$ exponentially as $\xi\rightarrow \infty$. So, whenever $F$ has two simple zeros $m, M$ and $F(\phi)>0$ for $m<\phi <M$ we have the periodic solutions. If $F$ has a double zero $m$, a simple zero $M$, and $F(\phi)>0$ for $m<\phi <M$, then there are solutions with decay. To $F$ with a simple zero at $m$ and $F(\phi)>0$ for $m<\phi$, then no bounded solution $\phi$ exists.

If we include all the functions $\phi\in H^{1}_{loc}(\R)$ that are solutions of (\ref{185}), we will expand the class of solutions, so that they will have another qualitative behavior. If $\phi$ approaches a simple pole $\phi =c$ of $F$, then, if $\phi(\xi_{0}) = c$,
			\begin{equation}
			\label{funcao com polo}
			\phi(\xi) -c = \alpha |\xi-\xi_{0}|^{\frac{2}{3}} + O((\xi-\xi_{0})^{\frac{4}{3}})\ \ \text{as}\ \ \xi\rightarrow \xi_{0}
			\end{equation}
for some constant $\alpha$, so cusped solutions occur. Also, peakons occur when the evolution of $\phi$ according to (\ref{185}) suddenly changes direction, that is, $\phi^{'}\mapsto -\phi^{'}$.

			Applying the above discussion to the particular case
				\begin{equation}
				\label{188}
				F(\phi) = \frac{P(\phi)}{c-\phi},\ \ \text{with}\ \ P(\phi) = \phi^{2}\left(c-\frac{\phi^{2}}{2}\right) +a\phi +d , 
				\end{equation}
we can classify all the bounded solutions of (\ref{184}). The proof is basically an inspection of all possible distributions of zeros and poles of $F$. Note that $P(\phi)$ is a fourth-degree polynomial with real coefficients, then it has one, two or four real zeros. Moreover,

			To prove Theorem \ref{classificacao3} suppose that $P(\phi)$ has only a real zero. This zero has double multiplicity because the polynomial $P(\phi)$ is fourth-degree. Let $m\in\R$ this double zero and $z,\overline{z}$ the others zeros in $\C\setminus\R$. For this distribution of zeros, we can write
			$$P(\phi) = -\frac{1}{2}(\phi -m)^{2}(\phi -z)(\phi -\overline{z})$$
and so, we obtain $F^{'}(m) = 0$ and
				$$F^{''}(m) = - \frac{|m-z|^{2}}{c-m}.$$
Thus, we can see that the solution $\phi$ of the mCH equation satisfies $(\ref{eq exp})$ to $c<m< \phi$ and, therefore, $\phi$ is not bounded. If $\phi <m <c$, then $- F^{''}(m) >0 $ and a similar result is obtained.

			Suppose that there are two simple real zeros, say $m$ and $M$, and we going to prove Theorem \ref{classificacao1}. We write for $z\in\C\setminus\R$
					$$P(\phi) = \frac{1}{2}(M-\phi)(\phi -m)(\phi -z)(\phi -\overline{z})$$ 
and comparing the coefficients of (\ref{188}) we obtain that $\Re(z)= -(M+m)/2$ with $z\in\C\setminus\R$ satisfying the hyperboloid of two sheet in $(m, M, \Im(z))$,
				\begin{equation}
				\label{hiperboloide}
				\frac{5}{4}m^{2} +\frac{5}{4}M^{2} -\Im(z)^{2} + \frac{3}{2}mM -2c = 0.
				\end{equation}
			%$$\left\{\begin{array}{ll}
			%				M+ 2\Re(z)+m = 0&\\
			%				2\Re(z)(M+m) +mM +|z|^{2} = 4k-2c&\\
			%				|z|^{2}(M+m) + 2\Re(z)mM = 2a&\\
			%				m|z|^{2}M = -2d&
			%				\end{array}\right.$$
%and then $\Re(z)= -(M+m)/2$ with $z\in\C\setminus\R$ satisfying the hyperboloid of one sheet $m^{2} + M^{2} -|z|^{2} + mM +4k -2c = 0$.
In addition, we have
						$$F^{'}(m) = \frac{1}{2}\cdot\frac{(M-m)|m-z|^{2}}{c-m}$$
and, from (\ref{eq polinomial}), we deduce from that there are periodic solutions $C^{\infty}$ for (\ref{184}) if and only if $m<M<c$ or $c<M<m$. Checking the different cases it is shown that these are all the $C^{\infty}$ bounded solutions of (\ref{184}) for $P(\phi)$ with two simple real zeros. Note that these solutions never touch the line $\phi =c$. Thus, the interval $E$ in (\ref{184}) must be the whole real line. In particular, these $C^{\infty}$ solutions can never be glued together with another wave.

For the case of bounded solutions to (\ref{184}) for which $\phi\rightarrow c$ at a finite end-point of $E$, gluing two of these together will give rise to a non-smooth wave. Periodic peakons occur if we leave $m<M=c$. In fact, consider $\phi$ a solution such that $m\leqslant\phi\leqslant M$. If $\phi$ is decreasing it will reach $m$ and turn back up to reaches $M$. $\phi$ can not stop or turn back anywhere because $\phi$ has not singularities at points where $\phi \neq c$. But note that $\phi$ increases until reaches $\phi = c$ and it can make a sudden turn at this point. This suddenly change of direction of $\phi$ makes $\phi^{'}$ have a jump which gives rise to a peak. To see how periodic cuspons happen, consider $m< c <M$. In this case, $F(\phi)$ has a simple zero at $m$, a simple pole at $c$ and $F(\phi)>0$ for $m<\phi<c$. Thus, from (\ref{funcao com polo}) we have
					$$\phi^{'}(\xi) = \left\{\begin{array}{ll}
																	\frac{2}{3}\alpha |\xi-\xi_{0}|^{-\frac{1}{3}} + O((\xi-\xi_{0})^{\frac{1}{3}})\ \ \text{as}\ \ \xi\downarrow \xi_{0}&\\
																	-\frac{2}{3}\alpha |\xi-\xi_{0}|^{-\frac{1}{3}} + O((\xi-\xi_{0})^{\frac{1}{3}})\ \ \text{as}\ \ \xi\uparrow \xi_{0},&
					\end{array}\right.$$
for some constant $\alpha$. In particular, the solution $\phi$ has a cusp.

			These are all nontrivial bounded solutions of (\ref{184}) for this distribution of zeros of the polynomial $P(\phi)$ and correspond to items $(a)-(c)$ and $(a^{'})-(c^{'})$ of Theorem \ref{classificacao1}.

			Finally, let's determine which solutions of (\ref{184}) have the same constant $a$ to be glued into a function satisfying $(TW3)$.	We can find the integration constant $a$, corresponding to a solution $\phi$, identifying the coefficients of $\phi$ in $(\ref{188})$. Recalling that $\Re(z)=-(M+m)/2$, we obtain
								$$2a = (M+m)(|z|^{2} - mM),$$
where $z\in\C\setminus\R$ satisfies the equation $(\ref{hiperboloide})$. Moreover, the quadric (\ref{hiperboloide}) can be rewritten, less than a change of variables, in its reduced form
							$$\frac{m^{2}}{\alpha^{2}} - \frac{M^{2}}{\beta^{2}} - \frac{\Im(z)^{2}}{\gamma^{2}} = 1,$$
with $\alpha = \sqrt{-2c}$, $\beta = 2\sqrt{-c}$ and $\gamma = \sqrt{-c}$, for $c< 0$.		
			
			Thus, we obtain all bounded functions satisfying $(TW1)$ and $(TW2)$ for $P(\phi)$ with two simple real zeros by gluing of all possible ways the solutions corresponding to the same hyperboloid of two sheet. Note that the waves resulting of this approach are exactly the waves of Theorem \ref{classificacao1}.
			
			To prove Theorem \ref{classificacao2}, assume that there are four real zeros, say $m$, $M$, $z$, and $r$, and write 
					$$P(\phi) = \frac{1}{2}(M-\phi)(\phi -m)(\phi -z)(\phi -r).$$
Analogously to the previous case we obtain $z= -m-r-M$ with $r$ satisfying the ellipsoid in $(m, M, r)$,
				\begin{equation}
				\label{elipsoide}
				r^{2} + m^{2} + M^{2} +rm +rM + mM -2c = 0,
				\end{equation}
			%$$\left\{\begin{array}{ll}
			%				M+r+z+m = 0&\\
			%				rM + zM +mM +zr +mr + mz = 4k-2c&\\
			%				zrM + mrM + mzM + mzr = 2a&\\
			%				mzrM = -2d&
			%				\end{array}\right.$$
					$$F^{'}(m)=\frac{1}{2}\cdot\frac{(M-m)(m-z)(m-r)}{c-m}$$
and 
					$$F^{''}(m) = \frac{(M-m)[(m-r) + (m-z)] - (m-z)(m-r)}{c-m} + \frac{(M-m)(m-z)(m-r)}{(c-m)^2}.$$
Also, checking the different cases we conclude that there are periodic solutions $C^{\infty}$ for (\ref{184}) if and only if $z<r<m<M<c$ or $c<M<m<r<z$. A bounded solution $\phi$ with $\phi\downarrow m$ when $\xi\rightarrow\pm\infty$, exists for $z<r=m<M<c$. A solution $\phi\uparrow m$ when $\xi\rightarrow\pm\infty$, exists for $c<M<m=r<z$. A solution $\phi$ where both above cases occur, that is, $\phi\uparrow M$ when $\xi\rightarrow +\infty$ and $\phi\downarrow m$ when $\xi\rightarrow -\infty$, happens for $z=m<M=r<c$ or to $c<r=M<m=z$ with $\phi\downarrow M$ and $\phi\uparrow m$ when $\xi\rightarrow \pm\infty$.

Moreover, for bounded solutions to (\ref{184}) for which $\phi\rightarrow c$ at a finite end-point of $E$ we have, for example, periodic peakons if $z<r<m<M=c$ and cuspons with decay if we consider $z<r=m< c <M$. In general, the different possibilities for bounded solutions of (\ref{184}) for $P(\phi)$ with four real zeros are those presented in $(a)-(g)$ and $(a^{'})-(g^{'})$ of Theorem \ref{classificacao2}.

Again to determine which solutions of (\ref{184}) have the same integration constant $a$ we identify the coefficients of $\phi$ in (\ref{188}) obtaining
			$$2a = -(M+m)r^{2} - (m+r)M^{2} - (M+r)m^{2} -2mrM$$
where $z=-m-r-M $ and $r\in\R$ satisfies equation $(\ref{elipsoide})$. So, (\ref{elipsoide}) can be rewrite in its reduced form
							$$\frac{m^{2}}{\alpha^{2}} + \frac{M^{2}}{\beta^{2}} + \frac{r^{2}}{\gamma^{2}} = 1,$$
with $\alpha =\beta = 2\sqrt{c}$ and $\gamma = \sqrt{c}$, for $c> 0$. Gluing of all possible ways the solutions corresponding to the same ellipsoid we obtain that all bounded functions satisfying $(TW1)$ and $(TW2)$ for this case are the waves of Theorem \ref{classificacao2}.
			
			To finalize the proof of Theorems \ref{classificacao1} and \ref{classificacao2} is enough to show that these waves belong to $H^{1}_{loc}(\R)$ and satisfy $(TW3)$. This fact is ensured by Lemma \ref{lema importante} but in order to prove this lemma we need to establish some technical results and notation.
						
			\begin{definition}
A function $f:I\longrightarrow\R$, $I=[a,b]\subset\R$, is said to be of bounded variation if
						$$\sup\sum_{i=1}^{N}|f(t_{i})-f(t_{i-1})|<+\infty ,$$
where the supremum is taken over all natural number $N$ and all choices of partitions $\{t_{i}\}_{i=1}^{N}$ such that $a=t_{0}< t_{1}< \cdot\cdot\cdot < t_{N} = b$. $BV(I)$ denote the set of functions $f:I\longrightarrow\R$ of bounded variation on $I$. $BV_{loc}(\R)$ denote the space of functions $f:\R\longrightarrow\R$ of bounded variation for all compact intervals $I\subset\R$.
\end{definition}

\begin{definition}
A function $f:X\longrightarrow\R$ defined on some set $X\subset\R$ is an $N$-function if $f$ maps sets of measure zero to sets of measure zero.
\end{definition}

\begin{lemma}
\label{lema0}
A function $f:I\longrightarrow\R$, $I=[a,b]\subset\R$, is absolutely continuous if and only if $f\in W^{1,1}_{loc}(I)$.
\end{lemma}

\begin{lemma}
\label{lema6}
Let $f:I\longrightarrow\R$, $I=[a,b]\subset\R$, be a continuous function of bounded variation. Then, $f$ is absolutely continuous if and only if $f$ is an $N$-function.
\end{lemma}
		
\begin{lemma}
\label{lema7}
If $f:I\longrightarrow\R$, $I=[a,b]\subset\R$, is continuous and $|f|$ is absolutely continuous, the $f$ is absolutely continuous.
\end{lemma}			
		
		The proof of Lemmas \ref{lema0}$-$\ref{lema7} can be found on book \cite{Rudin}.
		
\begin{lemma}
\label{lema importante}
Any bounded function $\phi$ satisfying (TW1) and (TW2) belongs to $H^{1}_{\text{loc}}(\R)$ and satisfies (TW3).
\end{lemma}
\begin{proof}
			Consider $E_{i}$ and $C$ as in $(TW1)$. We initially prove that $\phi$ is an $N$-function. To get this, we write
						$$\phi(N) = \phi(N\cap C) \cup \left(\bigcup_{i=1}^{+\infty} \phi(N\cap E_{i})\right)$$
for any set $N$ of null measure and then we use Lemma \ref{lema6} for $\phi$ restricted to $N\cap E_{i}$ and also observe that $\phi(N\cap C) = \{c\}$.

			Now, we will show that $\phi\in BV_{\text{loc}}(\R)$. Consider $\{i_{k}\}_{k=1}^{+\infty}$, or simply $\{j\}_{j=1}^{+\infty}$, a subsequence and $I$ a compact subinterval such that
			$$I\subset C \cup \left(\bigcup_{j=1}^{+\infty} E_{j}\right)\ \ \text{and}\ \ \sum_{j=1}^{+\infty}|E_{j}|\ <+\infty,$$
where $|E_{j}|$ denote the length of the interval $E_{j}$.

			Suppose, for example, the polynomial $P_{j}(\phi) = \phi^{2}(c-\phi^{2}/2) +a\phi + d_{j}$ with four real zeros, say $z_{j}, r_{j}, m_{j}$ and $M_{j}$, and the case $z_{j} < r_{j} \leqslant m_{j}\leqslant \phi < c \leqslant M_{j}$. From $(\ref{178})$, we see that
			\begin{eqnarray}
			\label{equacao 30}
			|E_{j}| = 2\dis\int_{m_{j}}^{c} \frac{\sqrt{c-\phi}}{\sqrt{\Big{|}\phi^{2}\left(c-\frac{\phi^{2}}{2}\right) +a\phi + d_{j}\Big{|}}}\ d\phi.
			\end{eqnarray}
			
			Note that $P_{j}(c)\geqslant 0$ for all $j$. Moreover, the derivative of $P_{j}$ is independent of $j$ and since $P_{j}$ has four zeros we have that the $b_{j}$'s form a bounded set. Therefore, we obtain
						\begin{equation}
						\label{equacao 31}
						0\leqslant P_{j}(\phi) \leqslant D(c-m_{j}),\ \ m_{j}\leqslant \phi\leqslant c,\ j\geqslant 1,
						\end{equation}
for some positive real constant $D$ independent of $j$.

		So, using $(\ref{equacao 30})$ and $(\ref{equacao 31})$ it follows that
								\begin{equation}
								\label{equacao 32}
								|E_{j}| \geqslant D_{0}(c-m_{j}),\ \ j\geqslant 1,
								\end{equation}
where $D_{0}=4/3\sqrt{D}$. Thus, $\phi\in BV_{\text{loc}}(\R)$ because the total variation of $\phi$ in $I$ is bounded above by
							$$2\sum_{j=1}^{+\infty}\left( c- \min_{\xi\in E_{j}}\phi(\xi)\right) = 2\sum_{j=1}^{+\infty} (c-m_{j}) \leqslant \frac{2}{D_{0}} \sum_{j=1}^{+\infty}|E_{j}|\ <+\infty .$$
							
		Moreover, $\phi$ is continuous on $I$ (see \cite{Lenells}, page 417) and so $\phi$ is continuous in $\R$. Hence, since $\phi$ is an $N$-function of bounded variation, we have by Lemma \ref{lema6} that $\phi$ is absolutely continuous. This means by Lemma \ref{lema0} that the distributional derivative of $\phi$ is in $L^{1}_{\text{loc}}(\R)$, $\phi$ is differentiable in the classical sense at almost all points and the two derivatives coincide. 
		
		To conclude that $\phi \in H^{1}_{loc}(\R)$, we will show that $\phi^{'}\in L^{2}_{loc}(\R)$. Using $(\ref{178})$ and $(\ref{equacao 31})$, we get
					$$ \int_{E_{j}} (\phi^{'})^{2}\ d\xi \leqslant \sqrt{D}\sqrt{c-m_{j}} \int_{m_{j}}^{c} \dis\frac{1}{\sqrt{c-\phi}}\ d\phi = 2\sqrt{D}(c-m_{j}),\ \ j\geqslant 1$$
and so, by $(\ref{equacao 32})$,
			$$\int_{E_{j}} (\phi^{'})^{2}\ d\xi \leqslant \frac{3}{2}D|E_{j}|,\ \ j\geqslant 1.$$
			
		Thus, we obtain
					$$\int_{\bigcup_{j=1}^{+\infty}E_{j}} (\phi^{'})^{2}\ d\xi \leqslant \frac{3}{2}D \sum_{j=1}^{+\infty} |E_{j}|\  <+\infty.$$
As Lemma \ref{lema2} implies that $\phi^{'}=0$ a.e. on $C=\phi^{-1}(c)$, it follows from the above inequality that $\phi \in H^{1}_{loc}(\R)$.
		
		Finally, it remains to prove that $(\phi - c)^{2}\in W^{2,1}_{loc}(\R)$. To get this, we will show that $|(\phi-c)\phi^{'}|$ is a bounded variation continuous $N$-function.
		
		Since $\phi^{'}\in L^{2}_{loc}(\R)$ and $\phi^{'}= 0$ a.e. on $C$, we have that $(\phi - c)\phi^{'} = 0$ a.e. on $C$. Therefore, from $(\ref{178})$,
						\begin{equation}
						\label{equacao 35}
						|(\phi -c)\phi^{'}| = \sqrt{P_{j}(\phi)}\sqrt{c-\phi},\ \ c\in E_{j},\ \ j\geqslant 1
						\end{equation}
and, then, $|(\phi -c)\phi^{'}|$ is smooth on any interval $E_{j}$ and $|(\phi -c)\phi^{'}| \rightarrow 0$ at finite endpoint of $E_{j}$. So, on each $E_{j}$, $|(\phi -c)\phi^{'}|$ is a symmetric function with respect to the  midpoint of $E_{j}$ with two humps. Thus, by $(\ref{equacao 31})$ and $(\ref{equacao 32})$, the total variation of $|(\phi -c)\phi^{'}|$ in $I$ is bounded by
		$$4\sum_{j=1}^{+\infty} \sup_{\xi\in E_{j}} |(\phi(\xi) -c)\phi^{'}(\xi)| = 4\sum_{j=1}^{+\infty} \sup_{\xi\in E_{j}}\sqrt{P_{j}(\phi)} \sup_{\xi\in E_{j}}\sqrt{c-\phi} \leqslant \frac{8}{3} \sum_{j=1}^{+\infty} |E_{j}|\ < +\infty.$$

			We see from $(\ref{equacao 35})$ that $|(\phi -c)\phi^{'}|$ is continuous. Now writing,
					$$|(\phi -c)\phi^{'}|(N) = |(\phi -c)\phi^{'}|(N\cap C) \cup \left(\bigcup_{j=1}^{+\infty} |(\phi -c)\phi^{'}|(N\cap E_{j})\right)$$
and using Lemma \ref{lema6} with $|(\phi -c)\phi^{'}|$ restricted to $N\cap E_{j}$, we obtain that $|(\phi -c)\phi^{'}|$ is an $N$-function.
			
			Again, we apply Lemma \ref{lema6} to get that $|(\phi -c)\phi^{'}|$ is absolutely continuous. Moreover, since $(\phi -c)\phi^{'}$ is smooth on $E_{j}$ and, for $\xi\in C$, $(\phi(\xi) -c)\phi^{'}(\xi) = |(\phi -c)\phi^{'}|(\xi) = 0$, we have that $(\phi -c)\phi^{'}$ is continuous. Thus, Lemma \ref{lema7} implies that $(\phi -c)\phi^{'}$ is absolutely continuous.
			
			We conclude, since $(\phi - c)^{2}$ is absolutely continuous, that
							$$\left[(\phi -c)^{2}\right]^{'} = 2(\phi - c)\phi^{'}\ \in W^{1,1}_{loc}(\R)$$
and, consequently, $(\phi -c)^{2}\in W^{2,1}_{loc}(\R)$.
\end{proof}	
		
\begin{remark}
In Theorems \ref{classificacao1} and \ref{classificacao2} the solutions are defined on the whole real line. Here we restrict the solution to the interval between two crests in the case of periodic waves, and to the part left or right of the crest in the case of decaying waves. The interval $E$ is accordingly defined to be a finite or half-infinite interval.
\end{remark}

\section{Dependence on parameters}\label{regularity}

			In addition to Theorems \ref{classificacao1} and \ref{classificacao2}, it is possible to characterize the relation between the traveling waves of (\ref{mCH}) and the parameters $m$, $M$ and $c$. Namely, such traveling waves are continuously dependents on these parameters. 

\begin{theorem}
\label{teorema2}
Let $(m_{i}, M_{i}, c_{i})$, $i\geqslant 1$, and $(m,M,c)$ be such that there are corresponding traveling waves of (\ref{mCH}) according to $(a)-(c)$ of Theorem \ref{classificacao1} or $(a)-(g)$ of Theorem \ref{classificacao2}. Let $\phi_{i}$, $i=1,2,...$ and $\phi$ be these traveling waves translated so that they all have crests at $\xi =0$. If $(m_{i}, M_{i}, c_{i})\rightarrow (m,M,c)$, then $\phi_{i}\rightarrow \phi$ in $H^{1}_{loc}(\R)$. In particular, we have uniform convergence on compact sets.
\end{theorem}
\begin{proof}[Sketch of proof]
			Assume $(m_{i}, M_{i}, c_{i})$, $i\geqslant 0$ and $(m, M, c)$ as in the statement of Theorem \ref{classificacao1} or \ref{classificacao2} and $\phi_{i}$ and $\phi$ the correspond traveling waves with crests at zero.
			
			Note that to the mCH equation we obtain from (\ref{178}) and (\ref{179}) that $\phi$ and $\phi_{i}$ are given implicitly by
			\begin{equation}
			\label{6.2}
			\xi=\left\{\begin{array}{ll}
			\xi_{0} + \int_{\phi_{0}}^{\phi}\frac{dy}{\sqrt{F(y)}},\ \phi^{'} >0&\\
			\xi_{0} - \int_{\phi_{0}}^{\phi}\frac{dy}{\sqrt{F(y)}},\ \phi^{'} <0,&
			\end{array}\right.
			\end{equation}
where $\phi(\xi_{0})=\phi_{0}$. Moreover, by choosing $\xi_{0}=0$, we have $\phi_{0} = \max_{\xi\in\R}\phi(\xi) = \min\{M, c\}$ and
			$$\xi=	- \int_{\min\{M, c\}}^{\phi}\frac{dy}{\sqrt{F(y)}},\ \ 0<\xi<\frac{L}{2},$$
where $L$ is the period of $\phi$ given by
			$$L=2\int_{m}^{\min\{M, c\}}\frac{dy}{\sqrt{F(y)}} = 2\int_{m}^{\min\{M, c\}}\frac{\sqrt{c-y}}{\sqrt{P(y)}} \ dy .$$

			The rest of proof follows a similar line to that of Theorem 2 in \cite{Lenells}, by considering the different distributions of zeros to the polynomial $P(\phi) = \phi^{2}\left(c-\phi^{2}/2\right) +a\phi +d$ instead of $P(\phi)=\phi^{2}(c-\phi) +a\phi +d$ in the equations (6.4), (6.13) and (6.17) of \cite{Lenells}.
\end{proof}

			Also, Theorem \ref{teorema2} is true for waves satisfying $(a^{'})-(c^{'})$ and $(a^{'})-(g^{'})$ of Theorems \ref{classificacao1} and \ref{classificacao2}, respectively.

\section{Explicit formulas for peakons}\label{explicit peakons}
			Suppose that $\phi$ is a wave corresponding to $(d)-(e)$ of Theorem \ref{classificacao2}. In this section we determine explicit formulas for the peaked traveling waves.
			
			From (\ref{6.2}), we have
			\begin{equation}
			\label{7.1}
			|\xi-\xi_{0}| = \displaystyle\int_{\phi_{0}}^{\phi} \frac{dy}{\sqrt{F(y)}} = \sqrt{2}\displaystyle\int_{\phi_{0}}^{\phi} \frac{\sqrt{c-y}}{\sqrt{(M-y)(y-m)(y-z)(y-r)}}\ dy.
			\end{equation}
			
Making $\phi = m+ (M-m)\sin^{2}(\theta)$, we rewrite (\ref{7.1}) by
			\begin{equation}
			\label{7.2}
			|\xi-\xi_{0}| =\frac{2\sqrt{2}}{\sqrt{M-m}} \displaystyle\int_{\theta_{0}}^{\theta} \frac{\sqrt{A -\sin^{2}(t)}}{\sqrt{B+\sin^{2}(t)}\sqrt{C+\sin^{2}(t)}}\ dt
			\end{equation}
where $$A=\frac{c-m}{M-m},\ \ B=\frac{m-r}{M-m}\ \ \text{and}\ \ C=\frac{m-z}{M-m}.$$

The period of $\phi$ is given by
\begin{eqnarray}
\label{eq periodo}
L &=& 2\sqrt{2}\displaystyle\int_{m}^{\min\{M,c\}} \frac{\sqrt{c-y}}{\sqrt{(M-y)(y-m)(y-z)(y-r)}}\ dy\nonumber\\
&=& 4\frac{\sqrt{2}}{\sqrt{M-m}}\displaystyle\int_{\theta_{0}}^{\theta_{\text{max}}} \frac{\sqrt{A -\sin^{2}(t)}}{\sqrt{B+\sin^{2}(t)}\sqrt{C+\sin^{2}(t)}}\ dt .
\end{eqnarray}

Assume $\phi$ a periodic peakon solution as in item $(d)$ of Theorem \ref{classificacao1}. So, $A=1$ and $B,C>0$, with $B\neq C$. From (\ref{7.2}), we have
			\begin{eqnarray*}
			|\xi-\xi_{0}| &=& \frac{2\sqrt{2}}{\sqrt{M-m}}\displaystyle\int_{\theta_{0}}^{\theta} \frac{\cos(t)}{\sqrt{B+\sin^{2}(t)}\sqrt{C+\sin^{2}(t)}}\ dt\\
			&=& \frac{2\sqrt{2}}{\sqrt{M-m}}\left[- \frac{i}{\sqrt{C}}\cdot F\left( \text{arcsin}\left( \frac{i}{B}\sin(t)\right)\ ; \sqrt{\frac{B}{C}}\right) \right]\Bigg{|}_{\theta_{0}}^{\theta},
			\end{eqnarray*}
where $F$ denotes the elliptic integral of the first kind (see Appendix), since $0 < B/C< 1$.

Choosing $\xi_{0}$ such that $\phi(\xi_{0}) = m$, that is, $\xi_{0}$ is a trough, we get $\theta_{0} = 0$. Moreover, since $M=c$ we obtain $\sin(\theta) = \sqrt{\phi -m}/\sqrt{c-m}$. Thus
			$$|\xi-\xi_{0}| = -\frac{i}{D_{1}} \cdot F\left( i\cdot\text{arcsinh}\left( \frac{1}{B} \sqrt{\frac{\phi -m}{c-m}} \right)\ ; k\right),$$
where $D_{1} = \sqrt{C}\sqrt{c-m}/2\sqrt{2}$ and $k=\sqrt{B/C}$ is the module of the elliptical integral (see Appendix).

So, by properties of the elliptic integral of the first kind \cite{BF}, we obtain
				$$\sin\left(i\cdot\text{arcsinh}\left( \frac{1}{B} \sqrt{\frac{\phi -m}{c-m}} \right)\right) = \text{sn}(D_{1}|\xi-\xi_{0}|i\ ;k),$$
where $sn$ is a Jacobian elliptic function (see Appendix).

Therefore, solving this equation we conclude that
					$$\phi(\xi) = m + D_{2}\text{tn}^{2}(D_{1}|\xi-\xi_{0}|\ ;k'),\ \ |\xi -\xi_{0}| \leqslant \frac{L}{2},$$
where
					$$L= D_{3}\cdot F\left(\text{arcsin}\left(\sqrt{\frac{1/B^{2}}{1+ 1/B^{2}}}\right)\ ; k'\right)$$
is the period of $\phi$ obtained from (\ref{eq periodo}), $D_{2} = B^{2}(c-m)$, $D_{3}= 2/D_{1}$, $k'^{2}= 1 -k^{2}$ and $tn$ is a Jacobian elliptic function (see Appendix). This is the explicit formula to the periodic peakons of the mCH equation.  
				
If $\phi$ is a peakon with decay as in $(e)$ of Theorem \ref{classificacao2}, then $A=1$, $B=0$ and $C>0$. So, (\ref{7.2}) gives us
				$$|\xi-\xi_{0}| =\frac{2\sqrt{2}}{\sqrt{M-m}} \displaystyle\int_{\theta_{0}}^{\theta} \frac{\cos(t)}{\sin^{2}(t)\sqrt{C+\sin^{2}(t)}}\ dt = \frac{2\sqrt{2}}{\sqrt{C}\sqrt{M-m}}\Big{[}-\text{arccsch}|t|\Big{]}\Big{|}_{\frac{\sin(\theta_{0})}{\sqrt{C}}}^{\frac{\sin(\theta)}{\sqrt{C}}}.$$

Choosing $\xi_{0}$ such that $\phi(\xi_{0}) = c$, that is, $\xi_{0}$ at the peak, we get $\sin(\theta_{0}) = \sin(\theta_{\text{max}}) = 1$. Moreover, since $M=c$ we obtain $\sin(\theta) = \sqrt{\phi -m}/\sqrt{c-m}$. Thus
				$$|\xi-\xi_{0}| = -\frac{2\sqrt{2}}{\sqrt{C}\sqrt{c-m}} \left[ \text{arccsch}\left|\frac{\sqrt{\phi -m}}{\sqrt{C}\sqrt{c-m}}\right| - \text{arccsch} \left(\frac{1}{\sqrt{C}}\right) \right].$$
				
Now, we use the identity
			$$\text{arccsch}(x) = \ln\left(\frac{1}{x} + \sqrt{\frac{1}{x^{2}} +1}\right) = \ln\left(\frac{1 +\sqrt{1 + x^{2}}}{x}\right);\ \ x\in\R,\ x\neq 0$$
in the above equation to obtain
						$$ -\frac{\sqrt{C}\sqrt{M-m}|\xi-\xi_{0}|}{2\sqrt{2}} = \ln\left(\frac{\frac{\sqrt{C}\sqrt{c-m}}{\sqrt{\phi -m}}\left( 1+ \sqrt{1+ \frac{\phi -m}{C(c-m)}}\right) }{\sqrt{C}\left(1+ \sqrt{1+\frac{1}{C}} \right)}\right).$$

Thus,
				$$\phi(\xi) = m +\frac{D_{4}e^{-D_{5}|\xi-\xi_{0}|}}{(D_{6}e^{-D_{5}|\xi-\xi_{0}|} -1)^{2}}$$
is an explicit expression for the peakons with decay, where $D_{4} = 4C^{2}(c-m)(1+\sqrt{1+1/C})^{2}$, $D_{5} = \sqrt{C}\sqrt{c-m}/\sqrt{2}$ and $D_{6}= C(1+\sqrt{1+1/C})^{2}$.

\section{Acknowledgment}
The first author was supported by CAPES/Brazil.

\section{Appendix}
In this appendix we will talk about some concepts used so far without further explanation. In accordance with \cite{BF}, we start setting the \textit{normal elliptic integral of the first kind}
			$$F(\phi , k) = F(\phi\ |\ k^{2} ) = F_{k}(\phi) = \int_{0}^{y} \ \frac{dt}{\sqrt{(1-t^2)(1-k^2 t^2)}} = \int_{0}^{\phi} \frac{d\varphi}{\sqrt{(1 - k^{2}\sin^{2}\varphi)}},$$
where $y=\sin\phi$.

			The parameter $k$ is called the \textit{modulus of the elliptic integral} and $k'^{2} = 1-k^{2}$ its \textit{complementary modulus}, and both may take any real or imaginary value. Here we wish to take $0 < k^{2} < 1$. Moreover, the variable $\phi$ is called \textit{argument} and it is usually taken belonging to $\left[\frac{-\pi}{2} , \frac{\pi}{2}\right]$.

			In his algebraic form, the elliptic integral above is finite for all real (or complex) values of $y$, including infinity. When $\phi= \frac{\pi}{2}$, the integral $K(k) \dot{=} F \left(\frac{\pi}{2}, k \right)$ is said to be \textit{complete}.
			
			We define the \textit{Jacobian elliptic functions} using the inverse function of the elliptic integral of the first kind. This inverse function exists because that 
$$ u(y_{1}, k) \equiv u = \int_{0}^{y_{1}} \frac{dt}{\sqrt{(1-t^{2})(1-k^{2}t^{2})}} = \int_{0}^{\phi} \frac{d\varphi}{\sqrt{1 - k^2 \sin^2 \varphi}} = F(\phi, k),$$
is a strictly increasing function of the real variable $y_{1}$ and, in its algebraic form, this integral has the property of being finite for all values of $y_{1}$. This inverse $\phi = \mathrm{am}(u,k) = \mathrm{am}u$ is called \textit{amplitude function}.

		There are several Jacobian elliptic functions that can be seen in \cite{BF}, but here we will only define the functions $\textit{sn}$, $\textit{cn}$ and \textit{tn} as follows
			$$
			\begin{array}{lll}
			\text{sn}(u,k) = \text{sin}\ \mathrm{am}(u,k) = \text{sin} \phi,&&\\
			\text{cn}(u,k) = \text{cos}\ \mathrm{am}(u,k) = \text{cos} \phi,&&\\
			\text{tn}(u,k) = \dis\frac{\text{sin}\ \mathrm{am}(u,k)}{\text{cos}\ \mathrm{am}(u,k)} = \frac{\text{sin} \phi}{\text{cos} \phi}.&&
			\end{array}
			$$
These functions have a real period, namely $4 K(k)$, and some important properties summarized by the formulas given below.
$$
\begin{array}{lll}
\text{sn}(-u) = -\text{sn}(u), &&\ \ \text{sn}^2u + \text{cn}^2u = 1,\\
 \text{cn}(-u)=\text{cn}(u), &&\ \ -1 \leqslant \text{sn} u , \text{cn} u \leqslant 1,\\
 \text{tn}(-u) = -\text{tn}(u), &&\ \  \text{sn}(iu,k) = i\text{tn}(u,k').
\end{array}
$$
\\

{
\fontsize{10pt}{\baselineskip}\selectfont

}

%COMO COLOCAR FIGURA
%\begin{figure}[h]
%\centering\includegraphics[width=0.4\linewidth]{placeholder}
%\caption{Figure caption}
%\end{figure}



%% References


%\bibliography{sample.bib}






% \end{thebibliography}

\end{document}